\documentclass[leqno]{article}
\usepackage{geometry}
\usepackage{graphicx}	
\usepackage[cp1251]{inputenc}
\usepackage[english]{babel}
\usepackage{mathtools}
\usepackage{amsfonts,amssymb,mathrsfs,amscd,amsmath,amsthm}

\ifpdf
\usepackage{hyperref}
\else
\usepackage[dvipdfmx]{hyperref}
\fi

\usepackage{verbatim}

\usepackage{url}

\usepackage{stmaryrd}

\usepackage{cmap}

\providecommand\phantomsection{}

\def\thtext#1{
  \catcode`@=11
  \gdef\@thmcountersep{. #1}
  \catcode`@=12
}

\def\threst{
  \catcode`@=11
  \gdef\@thmcountersep{.}
  \catcode`@=12
}

\newcounter{nxt}

\theoremstyle{plain}
\newtheorem{thm}{Theorem}[section]
\newtheorem{prop}[thm]{Proposition}

\newtheorem{cor}[thm]{Corollary}
\newtheorem{lem}[thm]{Lemma}

\theoremstyle{definition}

\newtheorem{dfn}[thm]{Definition}
\newtheorem{examp}[thm]{Example}

\newtheorem{rk}[thm]{Remark}

 \catcode`@=11
 \def\.{.\spacefactor\@m}
 \catcode`@=12

\newcommand{\e}{\varepsilon}
\newcommand{\g}{\gamma}
\renewcommand{\l}{\lambda}
\newcommand{\D}{\Delta}

\newcommand{\s}{\sigma}

\renewcommand{\t}{\tau}

\newcommand{\R}{\mathbb{R}}

\newcommand{\0}{\emptyset}
\renewcommand{\:}{\colon}
\newcommand{\x}{\times}
\renewcommand{\c}{\circ}
\newcommand{\tom}{\rightrightarrows}
\newcommand{\sm}{\setminus}
\renewcommand{\ss}{\subset}
\renewcommand{\sp}{\supset}

\newcommand{\oper}[1]{\operatorname{#1}}
\newcommand{\diam}{\oper{diam}}
\newcommand{\dis}{\oper{dis}}
\newcommand{\id}{{\oper{id}}}
\newcommand{\GH}{\oper{\mathcal{G\!H}}}

\newcommand{\cB}{\mathcal{B}}
\newcommand{\cM}{\mathcal{M}}
\newcommand{\cP}{\mathcal{P}}
\newcommand{\cR}{\mathcal{R}}

\newcommand{\bR}{{\bar R}}
\newcommand{\bX}{\bar{X}}
\newcommand{\bY}{\bar{Y}}
\newcommand{\bcR}{\overline{\cR}}

\newcommand{\rom}[1]{{\em #1}}

\renewcommand{\)}{\rom)}

\let\oldem\emph
\renewcommand{\emph}[1]{\oldem{\textbf{#1}}}
\let\oldbibitem\bibitem
\renewcommand{\bibitem}[1]{\oldbibitem{#1}\renewcommand{\emph}[1]{\oldem{##1}}}

\begin{document}
 \title{Classical and continuous Gromov--Hausdorff distances}
 \author{K.~V.~Semenov, A.~A.~Tuzhilin}
 \maketitle

\begin{abstract}
Starting from the definition of the Gromov--Hausdorff distance via distortion of correspon\-den\-ces, we add the requirement of semicontinuity of each correspondence and its inverse. It turns out that in the case of lower semicontinuity we obtain the same classical Gromov--Hausdorff distance, while for upper semicontinuity we are able to prove coincidence with the classical one only in cases where the spaces are either totally bounded or boundedly compact.

\textbf{Keywords\/}: metric spaces, Hausdorff distance, Gromov--Hausdorff distance, set-valued maps, lower semicontinuous and upper semicontinuous set-valued maps.
\end{abstract}

\section*{Introduction}
\markboth{Classical and continuous Gromov--Hausdorff distances}{Introduction}
The famous Gromov--Hausdorff distance~\cite{Gromov1981,Gromov1999} measures the degree of non-isometry of metric spaces: for isometric spaces the distance is zero, and the more ``dissimilar'' the spaces are, the larger this distance is. In the classical definition of the Gromov--Hausdorff distance, additional structures that metric spaces may be equipped with are not taken into account. Even the topology induced by the metric is ignored by this distance. In~\cite{LimMemoliSmith}, a modification of the Gromov--Hausdorff distance that takes continuity into account was proposed. It was noted that when comparing spheres in Euclidean space equipped with the standard intrinsic metric using the classical Gromov--Hausdorff distance, the result differs from comparison using the continuous analogue of this distance. In~\cite{LeeMorales}, another version of the continuous Gromov--Hausdorff distance was proposed for comparing dynamical systems and solutions of partial differential equations. However, the approach of these authors leads to a significant complication of the technique, since their version of the distance does not satisfy the triangle inequality.

We consider several modifications of the classical Gromov--Hausdorff distance that take continuity into account in one way or another. Namely, in the definition of the Gromov--Hausdorff distance via distortion of correspondences, we will additionally require that these correspondences and their inverses, which are essentially surjective set-valued maps, be either lower semicontinuous or upper semicontinuous. We will show that under additional assumptions the result is the same Gromov--Hausdorff distance. Thus, we narrow the family of correspondences, which, we hope, may in some cases simplify the computation of the classical Gromov--Hausdorff distance between specific classes of metric spaces.

\section{Basic definitions and preliminary results}
\markboth{Classical and continuous Gromov--Hausdorff distances}{\thesection.~Basic definitions and preliminary results}
We provide the definitions and results necessary for what follows; details can be found in~\cite{BurBurIva}, see also~\cite{TuzLect}. For convenience, the distance between points $x$ and $y$ of a metric space $X$ will be denoted by $|xy|$.

Let $X$ be an arbitrary metric space, $x\in X$, and $r>0$ and $s\ge0$ be real numbers. Denote by $U_r(x)=\bigl\{y\in X:|xy|<r\bigr\}$ and $B_s(x)=\bigl\{y\in X:|xy|\le s\bigr\}$ the \emph{open\/} and \emph{closed balls\/} centered at $x$ with radii $r$ and $s$, respectively. If $A$ and $B$ are nonempty subsets of $X$, then set $|xA|=|Ax|=\inf\bigl\{|xa|:a\in A\bigr\}$ and $|AB|=|BA|=\inf\bigl\{|ab|:a\in A,\,b\in B\bigr\}$. Next, define the \emph{open $r$-neighborhood\/} and \emph{closed $s$-neighborhood of a set $A$} by setting respectively
$$
U_r(A)=\bigl\{x\in X:|xA|<r\bigr\}\ \ \text{and}\ \ B_s(A)=\bigl\{x\in X:|xA|\le s\bigr\}.
$$

\begin{dfn}\label{dfn:H}
For nonempty subsets $A$ and $B$ of a metric space $X$, the \emph{Hausdorff distance from $A$ to $B$} is the quantity
\begin{multline*}
d_H(A,B)=\max\bigl\{\sup_{a\in A}|aB|,\,\sup_{b\in B}|Ab|\bigr\}=\\ =
\inf\bigl\{r>0:A\ss U_r(B)\ \text{and}\ U_r(A)\sp B\bigr\}=\\ =
\inf\bigl\{s\ge0:A\ss B_s(B)\ \text{and}\ B_s(A)\sp B\bigr\}.
\end{multline*}
\end{dfn}

If $X$ and $Y$ are isometric metric spaces, we denote this fact by $X\approx Y$.

\begin{dfn}\label{dfn:GH}
The \emph{Gromov--Hausdorff distance\/} between nonempty metric spaces $X$ and $Y$ is the quantity
$$
d_{GH}(X,Y)=\inf\bigl\{d_H(X',Y'):X',Y'\ss Z,\,X'\approx X,\,Y'\approx Y\bigr\},
$$
where the infimum is taken over all metric spaces $Z$ and all isometric embeddings of $X$ and $Y$ into $Z$.
\end{dfn}

Definition~\ref{dfn:GH} is ill-suited for concrete computations. There is an equivalent definition given in terms of correspondences. We provide the necessary concepts and results.

For an arbitrary set $Z$, denote by $\cP_0(Z)$ the family of all nonempty subsets of $Z$. In particular, $\cP_0(X\x Y)$ is the set of all nonempty relations between $X$ and $Y$.

\begin{dfn}
For any nonempty metric spaces $X$, $Y$ and a (nonempty) relation $\s\in\cP_0(X\x Y)$, its \emph{distortion $\dis\s$} is the quantity
$$
\dis\s=\sup\Bigl\{\bigl||xx'|-|yy'|\bigr|:(x,y),\,(x',y')\in\s\Bigr\}.
$$
\end{dfn}

The following results follow immediately from the definition of distortion.
\begin{prop}\label{prop:DisSubset}
If $\s,\t\in\cP_0(X\x Y)$ and $\t\ss\s$, then $\dis\t\le\dis\s$.
\end{prop}

\begin{prop}\label{prop:DisCompose}
If $\s\in\cP_0(X\x Y)$, $\t\in\cP_0(Y\x Z)$ and $\t\c\s\ne\0$, then $\dis(\t\c\s)\le\dis\s+\dis\t$.
\end{prop}

Recall that a \emph{set-valued map $f\:X\tom Y$ from a set $X$ to a set $Y$} is any map $f\:X\to\cP_0(Y)$, see details in~\cite{BorGelMyshOb}. In other words, such a map assigns to each point $x\in X$ a nonempty subset $f(x)\ss Y$. Note that some authors do not require $f(x)$ to be nonempty for each $x$. \textbf{In what follows we will identify a set-valued map $f\:X\to\cP_0(Y)$ with the subset $\cup_{x\in X}\bigl(\{x\}\x f(x)\bigr)$ in $X\x Y$}.

A set-valued map $f$ is \emph{surjective} if $\cup_{x\in X}f(x)=Y$. A relation $R\ss X\x Y$ between sets $X$ and $Y$ that is a surjective set-valued map $R\:X\to\cP_0(Y)$, $R\:x\mapsto R(x)=\{y\in Y:(x,y)\in R\}$, is called a \emph{correspondence}. The set of all correspondences between $X$ and $Y$ is denoted by $\cR(X,Y)$.

\begin{thm}[\cite{BurBurIva}]\label{thm:GH-metri-and-relations}
For any nonempty metric spaces $X$ and $Y$,
$$
d_{GH}(X,Y)=\frac12\inf\bigl\{\dis R:R\in\cR(X,Y)\bigr\}.
$$
\end{thm}

\begin{prop}\label{prop:IsomZeroDis}
Nonempty metric spaces $X$ and $Y$ are isometric if and only if there exists $R\in\cR(X,Y)$ such that $\dis R=0$. All such $R$ are exactly all isometries between $X$ and $Y$ \(we identify maps and their graphs\/\).
\end{prop}

\begin{prop}\label{prop:ClosedCorresp}
Let $X$ and $Y$ be nonempty metric spaces, $R\in\cR(X,Y)$ and $\bR\ss X\x Y$ be the closure of $R$. Then $\dis\bR=\dis R$.
\end{prop}

\begin{proof}
If $\dis R=\infty$, then also $\dis\bR=\infty$, since $R\ss\bR$, and in this case the equalities are proved.

Now suppose $\dis R<\infty$. Choose arbitrary $(x,y),\,(x',y')\in\bR$, then for each $\e>0$ there exist $(x_1,y_1),\,(x'_1,y'_1)\in R$ such that $|xx_1|<\e$, $|yy_1|<\e$, $|x'x'_1|<\e$, $|y'y'_1|<\e$, but then $|x_1x'_1|-2\e<|xx'|<|x_1x'_1|+2\e$ and $|y_1y'_1|-2\e<|yy'|<|y_1y'_1|+2\e$, therefore
$$
\bigl||xx'|-|yy'|\bigr|<\bigl||x_1x_1'|-|y_1y'_1|\bigr|+4\e,
$$
whence
$$
\dis\bR=\sup_{(x,y),\,(x',y')\in\bR}\bigl||xx'|-|yy'|\bigr|\le\sup_{(x_1,y_1),\,(x'_1,y'_1)\in R}\bigl||x_1x_1'|-|y_1y_1'|\bigr|+4\e=\dis R+4\e,
$$
and since $\e$ is arbitrary, $\dis\bR\le\dis R$. But $R\ss\bR$, so we also have the opposite inequality, and with it the desired equality.
\end{proof}

Denote by $\bcR(X,Y)$ the set of all closed correspondences $R\in\cR(X,Y)$. From Proposition~\ref{prop:ClosedCorresp} and Theorem~\ref{thm:GH-metri-and-relations} we immediately obtain the following result.

\begin{cor}\label{cor:ClosedCorresp}
For any nonempty metric spaces $X$ and $Y$,
$$
d_{GH}(X,Y)=\frac12\inf\bigl\{\dis R:R\in\bcR(X,Y)\bigr\}.
$$
\end{cor}

\begin{prop}\label{prop:EpsNeigToDis}
Let $X$ and $Y$ be nonempty metric spaces, $\e>0$, $R\in\cR(X,Y)$ and $R_\e=U_\e(R)\ss X\x Y$. Then $\dis R=\inf_{\e>0}\dis R_\e=\lim_{\e\to0}\dis R_\e$.
\end{prop}

\begin{proof}
If $\dis R=\infty$, then also $\dis R_\e=\infty$ for all $\e$, since $R\ss R_\e$, and in this case the equalities are proved. If $\dis R<\infty$, then $\dis R_\e$ is a monotonically increasing function of $\e$, so $\inf_{\e>0}\dis R_\e=\lim_{\e\to0}\dis R_\e$. The rest repeats verbatim the proof of Proposition~\ref{prop:ClosedCorresp}.
\end{proof}

We recall the necessary definitions from the theory of set-valued maps. Details can be found in~\cite{BorGelMyshOb}.

Let $X$ and $Y$ be topological spaces, then a set-valued map $f\:X\tom Y$ is called
\begin{itemize}
\item \emph{upper semicontinuous at a point $x\in X$} if for any neighborhood $U^{f(x)}$ of $f(x)$ there exists a neighborhood $V^x$ of $x$ such that for all $x'\in V^x$ we have $f(x')\ss U^{f(x)}$, i.e.\ $f(V^x)\ss U^{f(x)}$;
\item \emph{lower semicontinuous at a point $x\in X$} if for any open set $U\ss Y$ such that $f(x_0)\cap U\ne\0$, there exists a neighborhood $V^x$ of $x$ satisfying the following condition: for all $x'\in V^x$ we have $f(x')\cap U\ne\0$;
\item \emph{continuous at a point $x\in X$} if it is both upper and lower semicontinuous at $x$.
\end{itemize}
A set-valued map is \emph{upper semicontinuous, lower semicontinuous, or continuous} if the corresponding condition holds for all points $x\in X$.

\begin{rk}\label{rk:SingleValuedContAsMulti}
If $f$ is a single-valued map, each of the three definitions just given is equivalent to ordinary continuity at a point. Thus, every continuous single-valued map is upper semicontinuous, lower semicontinuous, and continuous in the sense of set-valued maps. In particular, the identity map, considered as a set-valued map, is continuous.
\end{rk}

\begin{prop}\label{prop:ContComposMult}
Let $f\:X\tom Y$ and $g\:Y\tom Z$ be set-valued maps that are either both upper semicontinuous, both lower semicontinuous, or both continuous. Then their composition is also such.
\end{prop}

\begin{proof}
First, let $f$ and $g$ be upper semicontinuous. Take an arbitrary point $x\in X$ and consider a neighborhood $U:=U^{g(f(x))}$, then for each point $y\in f(x)$ we have $g(y)\ss g(f(x))\ss U$, so $U$ is also a neighborhood of $g(y)$. Since $g$ is upper semicontinuous at $y$, there exists a neighborhood $V^y$ such that $g(V^y)\ss U$. Set $V=\cup_{y\in f(x)}V^y$, then $g(V)\ss U$, and $V$ is a neighborhood of $f(x)$, therefore, by the upper semicontinuity of $f$, there exists a neighborhood $W^x$ such that $f(W^x)\ss V$, whence $g\bigl(f(W^x)\bigr)\ss U$, so $g\c f$ is upper semicontinuous at $x$ and, by the arbitrariness of $x$, upper semicontinuous overall.

Now let $f$ and $g$ be lower semicontinuous. Take an arbitrary point $x\in X$ and consider an open set $U\ss Z$ such that $g\bigl(f(x)\bigr)\cap U\ne\0$. Since $g\bigl(f(x)\bigr)=\cup_{y\in f(x)}g(y)$, there exist $y\in f(x)$ for which $g(y)\cap U\ne\0$. Denote the set of such $y$ by $Y'$. Since $g$ is lower semicontinuous, for each $y\in Y'$ there exists a neighborhood $V^y$ such that for all $y'\in V^y$ we have $g(y')\cap U\ne\0$. Set $V=\cup_{y\in Y'}V^y$, then $V$ is an open set, and for all $y'\in V$ we have $g(y')\cap U\ne\0$ and, moreover, since $\0\ne Y'\ss f(x)$ and $Y'\ss V$, we have $f(x)\cap V\ne\0$. Since $f$ is lower semicontinuous, there exists a neighborhood $W^x$ such that for all $x'\in W^x$ we have $f(x')\cap V\ne\0$. But then $g\bigl(f(x')\bigr)\cap U\ne\0$ for all $x'\in W^x$, so $W^x$ satisfies the lower semicontinuity condition for $g\c f$ at $x$, and, by the arbitrariness of $x$, this map is lower semicontinuous overall.

For continuous set-valued maps, the proof is obtained by concatenating the two previous items.
\end{proof}

\begin{prop}\label{prop:RestrictSemicont}
Let $X$ and $Y$ be nonempty topological spaces, $Z\ss Y$ be a nonempty subset, and $f\:X\tom Y$ be an upper semicontinuous\/ \(lower semicontinuous, continuous\) set-valued map. Then the restriction $f|_Z$ is also such.
\end{prop}

\begin{proof}
In the definition of semicontinuity of $f$ at a point $x\in X$, it is asserted that if some property holds for $f(x)$, then there exists a neighborhood $U^x$ such that the same property also holds for $f(x')$ for all points $x'\in U^x$. But for $f|_Z$, semicontinuity requires the property to hold for all $x'\in U^x\cap Z$, which, of course, also holds.
\end{proof}

For metric spaces $X$ and $Y$, set
\begin{itemize}
\item $\cR_{us}(X,Y)=\{R\in\cR(X,Y):\text{$R,\,R^{-1}$ are upper semicontinuous}\}$;
\item $\cR_{ls}(X,Y)=\{R\in\cR(X,Y):\text{$R,\,R^{-1}$ are lower semicontinuous}\}$;
\item $\cR_{rc}(X,Y)=\{R\in\cR(X,Y):\text{$R,\,R^{-1}$ are continuous}\}$.
\end{itemize}

The following obvious inclusions hold:
\begin{equation}\label{eq:RInclusions}
\cR(X,Y)\sp\cR_{us}(X,Y)\cup\cR_{ls}(X,Y)\sp\cR_{us}(X,Y)\cap\cR_{ls}(X,Y)=\cR_{rc}(X,Y).
\end{equation}

Many of the facts presented below hold simultaneously for all the families of correspondences introduced above, so for brevity, when referring to one of these families, we will denote it by $\cR_{gen}$, i.e.\ $\cR_{gen}\in\{\cR,\cR_{us},\cR_{ls},\cR_{rc}\}$. For example, we prove the following general result.

\begin{prop}\label{prop:NonEmptyCorresp}
Each family $\cR_{gen}$ is nonempty.
\end{prop}

\begin{proof}
Note that the correspondence $R:=X\x Y$ is continuous, so $\cR_{rc}(X,Y)$, and together with it, by inclusions~(\ref{eq:RInclusions}), $\cR_{us}(X,Y)$, $\cR_{ls}(X,Y)$ and $\cR(X,Y)$ are always nonempty. 
\end{proof}

Furthermore, Remark~\ref{rk:SingleValuedContAsMulti} and Proposition~\ref{prop:ContComposMult} imply

\begin{cor}\label{cor:PreserveCompose}
For any nonempty metric spaces $X$, $Y$ and $Z$,
\begin{itemize}
\item if $\nu\:X\to Y$ is a homeomorphism\/ \(in particular, an isometry\/\), then
$$
\nu\in\cR(X,Y)\cap\cR_{us}(X,Y)\cap\cR_{ls}(X,Y)\cap\cR_{rc}(X,Y);
$$
\item for each $\cR_{gen}$,
$$
\cR_{gen}(Y,Z)\c\cR_{gen}(X,Y)\ss\cR_{gen}(X,Z).
$$
\end{itemize}
\end{cor}

We define the following modifications of the Gromov--Hausdorff distance, namely,
\begin{itemize}
\item $d^{us}_{GH}(X,Y)=\frac12\inf\bigl\{\dis R:R\in\cR_{us}(X,Y)\bigr\}$ is called the \emph{upper semicontinuous $GH$-distance};
\item $d^{ls}_{GH}(X,Y)=\frac12\inf\bigl\{\dis R:R\in\cR_{ls}(X,Y)\bigr\}$ is called the \emph{lower semicontinuous $GH$-distance};
\item $d^{rc}_{GH}(X,Y)=\frac12\inf\bigl\{\dis R:R\in\cR_{rs}(X,Y)\bigr\}$ is called the \emph{$R$-continuous $GH$-distance}.
\end{itemize}

\begin{examp}\label{examp:DiscBoth}
If $f\:X\tom Y$ is an arbitrary set-valued map, it is continuous. Indeed, for any point $x$, we can choose $U^x:=\{x\}$ as a neighborhood of $x$, and then for each point $x'\in U^x$ we have $f(x')=f(x)\ss f(x)$ and $f(x')\cap f(x)=f(x)\ne\0$, therefore for any open set $V\ss Y$ that either contains $f(x)$ or intersects $f(x)$, the same property also holds for all $x'\in U^x$, which proves the continuity of $f$. It follows that $\cR(X,Y)=\cR_{us}(X,Y)=\cR_{ls}(X,Y)=\cR_{rc}(X,Y)$, hence $d_{GH}(X,Y)=d^{us}_{GH}(X,Y)=d^{ls}_{GH}(X,Y)=d^{rc}_{GH}(X,Y)$.

Thus, for discrete $X$ and $Y$, all the Gromov--Hausdorff distances defined above are equal.
\end{examp}

From inclusions~(\ref{eq:RInclusions}) we immediately obtain the following result.
\begin{prop}\label{prop:GHleCGH}
For any metric spaces $X$ and $Y$,
\begin{itemize}
\item $d_{GH}(X,Y)\le\min\bigl\{d^{us}_{GH}(X,Y),\,d^{ls}_{GH}(X,Y),\,d^{rc}_{GH}(X,Y)\bigr\}$\rom;
\item $d^{rc}_{GH}(X,Y)\ge\max\bigl\{d^{us}_{GH}(X,Y),\,d^{ls}_{GH}(X,Y)\bigr\}$.
\end{itemize}
\end{prop}

\section{General properties of modified $GH$-distances}
\markright{\thesection.~General properties of modified $GH$-distances}

From the definitions of the distances under consideration, we immediately conclude that they are all nonnegative, equal to zero between each space and itself, and symmetric. It turns out that all these distances also satisfy the triangle inequality. By analogy with how we treated the sets of correspondences of different types, we denote by $d_{GH}^{gen}$ any of the variants of the Gromov--Hausdorff distance introduced above, i.e.\ $d_{GH}^{gen}\in\{d_{GH},d^{us}_{GH},d^{ls}_{GH},d^{rc}_{GH}\}$.

\begin{thm}\label{thm:AllGenPseudometrics}
All the Gromov--Hausdorff distances defined above are generalized pseudometrics.
\end{thm}

\begin{proof}
It suffices to verify the triangle inequality. Choose arbitrary nonempty metric spaces $X$, $Y$, $Z$ and show that $d_{GH}^{gen}(X,Z)\le d_{GH}^{gen}(X,Y)+d_{GH}^{gen}(Y,Z)$. If one of $d_{GH}^{gen}(X,Y)$ and $d_{GH}^{gen}(Y,Z)$ is infinite, the inequality holds. Now suppose both are finite. Choose an arbitrary $\e>0$ and find correspondences $R_1\in\cR_{gen}(X,Y)$ and $R_2\in\cR_{gen}(Y,Z)$ such that $2d_{GH}^{gen}(X,Y)\ge\dis R_1-\e$ and $2d_{GH}^{gen}(Y,Z)\ge\dis R_2-\e$, then $R:=R_2\c R_1\in\cR_{gen}(X,Z)$ by Corollary~\ref{cor:PreserveCompose}, and by Proposition~\ref{prop:DisCompose},
$$
2d_{GH}^{gen}(X,Z)\le\dis R\le\dis R_1+\dis R_2\le2d_{GH}^{gen}(X,Y)+2d_{GH}^{gen}(Y,Z)+2\e,
$$
from which we obtain the required result by the arbitrariness of $\e$.
\end{proof}

Furthermore, Corollary~\ref{cor:PreserveCompose} and Proposition~\ref{prop:IsomZeroDis} lead to

\begin{cor}\label{cor:IsomZedoGH}
Let $X$ and $Y$ be nonempty isometric metric spaces, then $d_{GH}^{gen}(X,Y)=0$.
\end{cor}

In what follows, by $\GH$ we denote the proper class in the sense of von Neumann--Bernays--G\"odel set theory~\cite{Mendelson, TBanach} consisting of all metric spaces considered up to isometry. From Corollary~\ref{cor:IsomZedoGH} it follows that all four variants of the Gromov--Hausdorff distance under consideration are well-defined on $\GH$. The proper class $\GH$ is called the \emph{Gromov--Hausdorff class}. The subclass of $\GH$ consisting of all bounded metric spaces is denoted by $\cB$, and the subclass of $\cB$ of all compact metric spaces is denoted by $\cM$. Note that the class $\cB$ is proper. The class $\cM$, however, is a set, called the \emph{Gromov--Hausdorff space}.

For a metric space $X$, denote by $\diam X$ its \emph{diameter\/}:
$$
\diam X=\sup\bigl\{|xy|:x,y\in X\bigr\}.
$$
For a one-point metric space we reserve the notation $\D_1$. If $\l>0$, then $\l X$ denotes the metric space obtained from $X$ by multiplying all distances by $\l$. If $X$ is bounded, set $0 X=\D_1$. Note that $\diam(\l X)=\l\diam X$.

In the following proposition, we show that the main properties of the ordinary Gromov--Hausdorff distance (see~\cite{BurBurIva}) also hold for its continuous analogues.

\begin{prop}\label{prop:GHelementProps}
For any $X,Y\in\GH$,
\begin{enumerate}
\item\label{prop:GHelementProps:1} $2d^{gen}_{GH}(\D_1,X)=\diam X$\rom;
\item\label{prop:GHelementProps:2} $2d^{gen}_{GH}(X,Y)\le\max\{\diam X,\diam Y\}$\rom;
\item\label{prop:GHelementProps:3} if the diameter of $X$ or $Y$ is finite, then $\bigl|\diam X-\diam Y\bigr|\le2d^{gen}_{GH}(X,Y)$\rom;
\item\label{prop:GHelementProps:4} if the diameter of $X$ is finite, then for any $\l\ge0$, $\mu\ge0$, $2d^{gen}_{GH}(\l X,\mu X)=|\l-\mu|\diam X$, from which it immediately follows that the curve $\g(t):=t\,X$ is a shortest path between any of its points, and the length of such a segment of the curve equals the distance between its endpoints\rom;
\item\label{prop:GHelementProps:5} for any $\l>0$, $d^{gen}_{GH}(\l X,\l Y)=\l\,d^{gen}_{GH}(X,Y)$, and if $X$ and $Y$ are bounded, the equality also holds for $\l=0$\rom;
\item\label{prop:GHelementProps:6} if $X$ and $Y$ are discrete metric spaces, then $d^{gen}_{GH}(X,Y)=d_{GH}(X,Y)$\rom;
\item\label{prop:GHelementProps:7} if $X_1,\,X_2,\ldots$ is a sequence of metric spaces converging with respect to $d^{gen}_{GH}$ to a metric space $X$, then it also converges with respect to $d_{GH}$ to this $X$.
\end{enumerate}
\end{prop}

\begin{proof}
(\ref{prop:GHelementProps:1}) Note that the set $\cR(\D_1,X)$ consists of exactly one correspondence $R:=\D_1\x X$, and since the remaining $\cR_{gen}(\D_1,X)$ contained in this set are always nonempty by Proposition~\ref{prop:NonEmptyCorresp}, they all coincide with $\{R\}$. It remains to note that $\dis R=\diam X$.

(\ref{prop:GHelementProps:2}) Since each $R\in\cR(X,Y)$ is contained in the correspondence $X\x Y$, by Proposition~\ref{prop:DisSubset} we have $\dis R\le\dis(X\x Y)$. It remains to note that $\dis(X\x Y)=\max\{\diam X,\diam Y\}$.

(\ref{prop:GHelementProps:3}) If $\diam X=\diam Y$, the inequality holds. Let $\diam X>\diam Y$, then the diameter of $Y$ is finite. Choose arbitrary $R\in\cR_{gen}(X,Y)$ and consider sequences $(x_n,y_n),\,(x'_n,y'_n)\in R$ such that $|x_nx'_n|\to\diam X$ as $n\to\infty$. Then, starting from some $n$, we have $|x_nx'_n|>\diam Y\ge|y_ny'_n|$, whence for all such $n$,
$$
\bigl||x_nx'_n|-|y_ny'_n|\bigr|=|x_nx'_n|-|y_ny'_n|\ge|x_nx'_n|-\diam Y\to(\diam X-\diam Y)\ \text{as $n\to\infty$},
$$
so $\dis R\ge\diam X-\diam Y=|\diam X-\diam Y|$. By the arbitrariness of $R$, we obtain the required result.

(\ref{prop:GHelementProps:4}) By item~(\ref{prop:GHelementProps:3}) we have $2d^{gen}_{GH}(\l X,\mu X)\ge\bigl|\diam(\l X)-\diam(\mu X)\bigr|=|\l-\mu|\diam X$. To prove the reverse inequality, let $R\:\l X\to\mu X$ be the identity map, then, by Corollary~\ref{cor:PreserveCompose}, this $R$ is contained in all four $\R_{gen}(X,X)$. It remains to note that $\dis R=|\l-\mu|\diam X$.

(\ref{prop:GHelementProps:5}) For each $R\in\cR(X,Y)$ denote by $R_\l$ the same correspondence, but considered as an element of $\cR(\l X,\l Y)$. Then $\dis R_\l=\l\dis R$, from which the statement of this item follows.

(\ref{prop:GHelementProps:6}) This follows from the fact that all maps between $X$ and $Y$, both single-valued and set-valued, are continuous.

(\ref{prop:GHelementProps:7}) This follows from the inequality $d_{GH}\le d^{gen}_{GH}$.
\end{proof}

\begin{rk}
Analysis of the facts we used in the proofs of Theorem~\ref{thm:AllGenPseudometrics} and Proposition~\ref{prop:GHelementProps} (except item~\ref{prop:GHelementProps:6}) allows us to obtain the following generalization. For each pair $(X,Y)$ of metric spaces, denote by $\cR_g(X,Y)$ a \emph{nonempty\/} subset of $\cR(X,Y)$ and require that the following properties hold:
\begin{itemize}
\item for any metric space $X$, the identity map $\id$ is contained in $\cR_g(X,X)$;
\item for any metric spaces $X$, $Y$ and $Z$, $\cR_g(Y,Z)\c\cR_g(X,Y)\ss\cR_g(X,Z)$;
\item for any metric spaces $X$, $Y$, $\cR_g(Y,X)=\bigl[\cR_g(X,Y)\bigr]^{-1}$.
\end{itemize}
Note that the first two conditions define a \emph{category} whose \emph{objects\/} are metric spaces and whose \emph{morphisms} are correspondences from $\cR_g(X,Y)$.

Define the Gromov--Hausdorff \emph{$g$-distance\/} by the formula
$$
d^g_{GH}(X,Y)=\frac12\inf\bigl\{\dis R:R\in\cR_g(X,Y)\bigr\}.
$$
Then $d^g_{GH}(X,Y)$ satisfies Theorem~\ref{thm:AllGenPseudometrics} and Proposition~\ref{prop:GHelementProps} (except possibly item~\ref{prop:GHelementProps:6}), where $\cR_{gen}$ is replaced by $\cR_g$.
\end{rk}

\section{Coincidence of the semicontinuous $GH$-distance with the classical one}
\markright{\thesection.~Coincidence of the semicontinuous $GH$-distance with the classical one}

We begin with the case of the lower semicontinuous distance.

\subsection{Lower semicontinuous $GH$-distance}

Let $f\:X\tom Y$ be a set-valued map of sets. The \emph{full preimage} of a set $D\ss Y$ is the set
$$
f_-^{-1}(D)=\bigl\{x\in X:f(x)\cap D\ne\0\bigr\}.
$$

\begin{thm}[\cite{BorGelMyshOb}]\label{thm:LowerSemicont}
Let $f\:X\tom Y$ be a set-valued map of topological spaces. Then $f$ is lower semicontinuous if and only if for any open set $U\ss Y$, the set $f_-^{-1}(U)$ is open in $X$.
\end{thm}

Recall that a map $f\:X\to Y$ of topological spaces is called \emph{open} if the image of every open set is open. In what follows, we will need the following lemma.

\begin{lem}\label{lem:CartesProjOpen}
Let $X$ and $Y$ be topological spaces, and $\pi_X\:X\x Y\to X$ be the standard projection. Then $\pi_X$ is an open map.
\end{lem}

\begin{thm}\label{thm:SemiContOpenNeib}
Let $X$ and $Y$ be topological spaces, and $R\in\cR(X,Y)$ be an open subset of $X\x Y$. Then $R\:X\tom Y$ and $R^{-1}\:Y\tom X$ are lower semicontinuous set-valued maps. In other words, $R\in\cR_{ls}(X,Y)$.
\end{thm}

\begin{proof}
We prove lower semicontinuity for $R$ (the proof for $R^{-1}$ is analogous). Choose an arbitrary open $U\ss Y$, then $R_-^{-1}(U)=\pi_X\bigl[(X\x U)\cap R\bigr]$, and since $X\x U$ is open and $R$ is open, $(X\x U)\cap R$ is open. By Lemma~\ref{lem:CartesProjOpen}, $\pi_X\bigl[(X\x U)\cap R\bigr]$ is open, so $R$ is lower semicontinuous by Theorem~\ref{thm:LowerSemicont}.
\end{proof}

\begin{cor}\label{cor:AllLsSemi}
Let $X$ and $Y$ be nonempty metric spaces, then $d_{GH}(X,Y)=d^{ls}_{GH}(X,Y)$.
\end{cor}

\begin{proof}
By Theorem~\ref{thm:SemiContOpenNeib}, for each $R\in\cR(X,Y)$, $\e>0$ and $R_\e:=U_\e(R)$ we have $R_\e\in\cR_{ls}(X,Y)$. On the other hand, by Proposition~\ref{prop:EpsNeigToDis}, $\inf_\e\dis R_\e=\dis R$, so
$$
2d_{GH}(X,Y)=\inf_{R\in\cR(X,Y)}\dis R=\inf_{R\in\cR(X,Y),\,\e>0}\dis R_\e\ge\inf_{R\in\cR_{ls}(X,Y)}\dis R=2d^{ls}_{GH}(X,Y).
$$
It remains to apply Proposition~\ref{prop:GHleCGH}.
\end{proof}

We now consider the case of the upper semicontinuous distance.

\subsection{Upper semicontinuous $GH$-distance}

Let $f\:X\tom Y$ be a set-valued map of sets. The \emph{small preimage} of a set $D\ss Y$ is the set
$$
f_+^{-1}(D)=\bigl\{x\in X:f(x)\ss D\bigr\}.
$$

\begin{thm}[\cite{BorGelMyshOb}]\label{thm:UpperSemicont}
Let $f\:X\tom Y$ be a set-valued map of topological spaces. Then $f$ is upper semicontinuous if and only if for any closed set $W\ss Y$, the set $f_-^{-1}(W)$ is closed in $X$.
\end{thm}

\begin{thm}\label{thm:CompHausd}
Let $X$ and $Y$ be nonempty compact topological spaces, with $X$ also Hausdorff, and let $f\:X\tom Y$ be a set-valued map whose graph $f\ss X\x Y$ is closed. Then $f$ is upper semicontinuous.
\end{thm}

\begin{proof}
Choose an arbitrary closed $W\ss Y$ and denote by $\pi_X\:X\x Y\to X$ the standard projection. Then
$$
f_-^{-1}(W)=\bigl\{x\in X:f(x)\cap W\ne\0\bigr\}=\bigl\{x\in X:\exists y\in W,\,y\in f(x)\bigr\}=\pi_X\bigl[(X\x W)\cap f\bigr].
$$
Since $X\x W$ is closed in $X\x Y$, $(X\x W)\cap f$ is also closed, and hence compact due to the compactness of $X\x Y$. Since the projection $\pi_X$ is continuous, and the continuous image of a compact set is compact, the set $f_-^{-1}(W)\ss X$ is compact and, therefore, closed because $X$ is Hausdorff. Thus, by Theorem~\ref{thm:UpperSemicont}, the set-valued map $f$ is upper semicontinuous.
\end{proof}

From Theorem~\ref{thm:CompHausd} we immediately obtain the following result.

\begin{cor}\label{cor:ClosedCorrespUpSemi}
Let $X$ and $Y$ be nonempty compact Hausdorff topological spaces. Then every closed correspondence $R\in\cR(X,Y)$, together with its inverse, is upper semicontinuous. In other words, $\bcR(X,Y)\ss\cR_{us}(X,Y)$.
\end{cor}

\begin{cor}\label{cor:CompactUpSemi}
Let $X$ and $Y$ be nonempty compact metric spaces. Then
$$
d_{GH}^{us}(X,Y)=d_{GH}(X,Y).
$$
\end{cor}

\begin{proof}
By Corollary~\ref{cor:ClosedCorrespUpSemi}, we have $d_{GH}(X,Y)\ge d^{us}_{GH}(X,Y)$. The opposite inequality is contained in Proposition~\ref{prop:GHleCGH}.
\end{proof}

We will need the following lemma.

\begin{lem}\label{lem:ClosedInCartOfSubsets}
Let $A\ss X$ and $B\ss Y$ be nonempty subsets of metric spaces, and $R\ss A\x B$ be closed in $A\x B$. Let $(a_n,b_n)\in R$ be a sequence such that $a_n\to a\in A$ and $b_n\to b\in B$. Then $(a,b)\in R$.
\end{lem}

\begin{proof}
Suppose the contrary, i.e., $(a,b)\not\in R$. Since $R$ is closed in $A\x B$, there exists $\e>0$ such that $\bigl(U_\e(a)\cap A\bigr)\x\bigl(U_\e(b)\cap B\bigr)$ does not intersect $R$. But for sufficiently large $n$, we have $a_n\in U_\e(a)\cap A$ and $b_n\in U_\e(b)\cap B$, a contradiction.
\end{proof}

\begin{thm}\label{thm:SubsertOfCompactMetrSp}
Let $Y\ss X$ be a nonempty subset of a compact metric space $X$. Then $d^{us}_{GH}(X,Y)=d_{GH}(X,Y)$.
\end{thm}

\begin{proof}
For arbitrary $R'\in\bcR(Y,X)$ and $\e>0$, extend $R'$ to $R\in\cR(\bY,X)$ as follows: for each point $y'\in\bY\sm Y$,
\begin{itemize}
\item consider an arbitrary sequence $y'_n\in Y$, $y'_n\to y'$;
\item for each $n$, choose $x'_n\in X$ such that $(y'_n,x'_n)\in R'$;
\item using the compactness of $X$, choose a subsequence from $x'_n$ converging to some $x'\in X$;
\item without loss of generality, assume that $x'_n\to x'$;
\item add all such obtained $(y',x')$ to $R'$.
\end{itemize}

Next, let $\bR\ss\bY\x X$ be the closure of $R$ in $\bY\x X$. By Corollary~\ref{cor:ClosedCorrespUpSemi}, the set-valued map $\bR\:\bY\tom X$ is upper semicontinuous. By Proposition~\ref{prop:RestrictSemicont}, the restriction $\bR|_Y$ is also upper semicontinuous.

We show that $\bR|_Y=R'$. Choose an arbitrary point $(y,x)\in\bR|_Y$, then there exists a sequence $(y_n,x_n)\in R$ converging to $(y,x)$. By construction, for each $y_n$ there are two possibilities:
\begin{itemize}
\item either $y_n\in Y$, then we remain it unchanged;
\item or $y_n\in\bY\sm Y$, but then, by construction, there exists $y'_n\in Y$ such that $(y'_n,x'_n)\in R'$ and $|y'_ny_n|<1/n$; in this case, we replace the pair $(y_n,x_n)$ with $(y'_n,x'_n)$.
\end{itemize}
Note that for the reconstructed sequence $(y_n,x_n)$ we have $(y_n,x_n)\in R'$ for all $n$, $y_n\to y$, $x_n\to x$. Since $y\in Y$ and $R'$ is closed, by Lemma~\ref{lem:ClosedInCartOfSubsets}, $(y,x)\in R'$.

Thus, we have shown that $\bcR(Y,X)\ss\cR_{us}(Y,X)$, which completes the proof.
\end{proof}

From Theorem~\ref{thm:SubsertOfCompactMetrSp} we immediately obtain the following result.

\begin{cor}\label{cor:DenseSubset}
Let $Y\ss X$ be a dense subset of a nonempty compact metric space. Then $d_{GH}^{us}(X,Y)=0$.
\end{cor}

\begin{cor}\label{cor:TotalBound}
Let $X$ and $Y$ be totally bounded metric spaces. Then $d^{us}_{GH}(X,Y)=d_{GH}(X,Y)$.
\end{cor}

\begin{proof}
Let $\bX$ and $\bY$ be the completions of $X$ and $Y$, respectively. Then $\bX$ and $\bY$ are compact, and $X$ and $Y$ are their dense subsets. It remains to use Theorem~\ref{cor:CompactUpSemi}, Corollary~\ref{cor:DenseSubset} and the triangle inequality for $d^{us}_{GH}$.
\end{proof}

\begin{cor}
Let $X$ and $Y$ be boundedly compact metric spaces\/ \(i.e., all closed balls in them are compact\/\) with $d_{GH}(X,Y)<\infty$. Then $d^{us}_{GH}(X,Y)=d_{GH}(X,Y)$.
\end{cor}

\begin{proof}
It suffices to choose an arbitrary closed correspondence $R\in\bcR(X,Y)$ with $\dis R<\infty$ and show that for any $x\in X$, the set-valued map $R$ is upper semicontinuous at $x$.

Consider arbitrary open and closed balls $U:=U_r(x)$ and $A:=B_r(x)$, respectively. Since $\dis R<\infty$, the set $R(A)$ is bounded, so there exists an open ball $V:=U_s(y)\ss Y$ such that $R(A)\ss V$. Set $B=B_s(y)$. Since $X$ and $Y$ are boundedly compact, $A$ and $B$ are compact. The set-valued map $R|_A$ maps the metric compact $A$ into the metric compact $B$. Since $R$ is closed in $X\x Y$, $R|_A=(A\x B)\cap R$ is closed in $A\x B$. By Theorem~\ref{thm:CompHausd}, the set-valued map $R|_A$ is upper semicontinuous.

Let $W\ss Y$ be an arbitrary open set in $Y$ containing $R(x)$. Since $R(x)\ss R(A)\ss V$, we have $R(x)\ss V\cap W\ss B$, and $V\cap W$ is an open set in $Y$ and, therefore, also an open set in the topology induced on $B$. Since $R|_A$ is upper semicontinuous, there exists a neighborhood $W^x$ in $A$ such that for all $x'\in W^x$, $R(x')=R|_A(x')\ss V\cap W$. Note that the set $W^x\cap U$ is open in $X$, so it is an open neighborhood of $x$ in $X$ as well. Thus, we have constructed a neighborhood in $X$ of $x$ such that for all points $x'$ in it, the set $R(x')$ is contained in an arbitrarily chosen neighborhood in $Y$ of $R(x)$, which proves the upper semicontinuity of the set-valued map $R$ at $x$.
\end{proof}

\phantomsection
\renewcommand\bibname{References}


\addcontentsline{toc}{chapter}{\bibname}

\begin{thebibliography}{99}
\bibitem{Gromov1981} M.~Gromov, Structures m\'etriques pour les vari\'et\'es riemanniennes. Edited by Lafontaine and Pierre Pansu, 1981.

\bibitem{Gromov1999} M.~Gromov, Metric structures for Riemannian and non-Riemannian spaces. Birkh\"auser, 1999.

\bibitem{LimMemoliSmith} S.~Lim, F.~Memoli and Z.~Smith, \emph{The Gromov–Hausdorff distance between spheres}. Geometry \& Topology, 2023, v. 27, N 9, pp. 3733--3800.

\bibitem{LeeMorales}  J.~Lee, C.\,A.~Morales, Gromov-Hausdorff Stability of Dynamical Systems and Applications to PDEs. Birkh\"auser/Springer, 2022.

\bibitem{BurBurIva} D.\,Yu.~Burago, Yu.\,D.~Burago, S.\,V.~Ivanov, \emph{A Course in Metric Geometry}, Graduate Studies in Mathematics {\bf 33} A.M.S., Providence, RI, 2001.

\bibitem{TuzLect}  \href{https://arxiv.org/pdf/2012.00756}{A.\,A.~Tuzhilin, \emph{Lectures on Hausdorff and Gromov-Hausdorff Distance Geometry}, 2020, ArXiv e-prints, arXiv:2012.00756.}

\bibitem{BorGelMyshOb} Yu.\,G.~Borisovich, B.\,D.~Gel'man, A.\,D.~Myshkis, V.\,V.~Obukhovskii, \emph{Multivalued mappings}, Journal of Soviet Mathematics, 1984, vol. 24, N 6, pp. 719--791.

\bibitem{Mendelson} E.~Mendelson, Introduction to Mathematical Logic. Moscow: Nauka, 1984.

\bibitem{TBanach} \href{https://arxiv.org/pdf/2006.01613}{T.~Banach, \emph{Classical set theory: theory of sets and classes}, 2023, ArXiv e-prints, arXiv:2006.01613[math.LO].}


\end{thebibliography}
\end{document}